\documentclass[12pt,twoside,a4paper,reqno]{amsart}
\usepackage[bookmarks=false]{hyperref}
\usepackage{amsmath,amssymb,verbatim}
\usepackage[all]{xy}
\usepackage{graphicx}

\date{\today}
\usepackage[latin1]{inputenc}
\usepackage[T1]{fontenc}

\def\End{{\rm End}}

\def\w{\wedge}

\def\dbar{\bar\partial}

\def\Cn{\C^n}

\def\Hom{{\rm Hom\, }}

\def\Im{{\rm Im\, }}

\def\O{{\mathcal O}}

\def\xk{X_k}

\def\ann{\text{ann}\,}

\def\1{\mathbf 1}
\def\lcm{\text{lcm}\;}

\def\m{{\mathfrak m}}

\def\be{\begin{equation}}
\def\ee{\end{equation}}

\def\N{{\mathbf N}}
\def\Z{{\mathbf Z}}
\def\R{{\mathbf R}}
\def\C{{\mathbf C}}
\def\F{{\mathbf F}}

\def\a{\mathfrak{a}}
\def\b{\mathfrak{b}}
\def\Ok{\mathcal{O}}

\DeclareMathOperator{\spann}{span}
\DeclareMathOperator{\hull}{hull}
\DeclareMathOperator{\sign}{sgn}
\DeclareMathOperator{\sgn}{sgn}
\DeclareMathOperator{\Vol}{Vol}

\newtheorem{thm}{Theorem}[section]
\newtheorem{lma}[thm]{Lemma}

\newtheorem{prop}[thm]{Proposition}

\theoremstyle{definition}

\newtheorem{df}[thm]{Definition}

\numberwithin{figure}{section}

\theoremstyle{remark}

\newtheorem{preremark}[thm]{Remark}
\newtheorem{preex}[thm]{Example}

\newenvironment{remark}{\begin{preremark}}{\qed\end{preremark}}
\newenvironment{ex}{\begin{preex}}{\qed\end{preex}}

\numberwithin{equation}{section}

\begin{document}

\title[Computing residue currents using comparison formulas]{Computing residue currents of monomial ideals using comparison formulas}

\date{\today}

\author{Richard L\"ark\"ang \& Elizabeth Wulcan}

\address{Department of Mathematics\\Chalmers University of Technology and the University of
G\"oteborg\\S-412 96 G\"OTEBORG\\SWEDEN}

\email{larkang@chalmers.se \& wulcan@chalmers.se}



\begin{abstract}
Given a free resolution of an ideal $\mathfrak{a}$ of holomorphic functions, one
can construct a vector-valued residue current $R$, which coincides with
the classical Coleff-Herrera product if $\mathfrak{a}$ is a complete
intersection ideal and whose annihilator ideal
is precisely ~$\mathfrak{a}$. 

We give a complete description of $R$ in the case when $\mathfrak{a}$ is an
Artinian monomial ideal and the resolution is the hull
resolution (or a more general cellular resolution). 
The main ingredient in the proof is a comparison formula for
residue currents due to the first author.

By means of this description, we obtain in the monomial case a
current version of a factorization of the fundamental cycle of $\a$
due to Lejeune-Jalabert. 
\end{abstract}

\maketitle
\section{Introduction}\label{intro}

With a regular sequence $f_1,\ldots, f_p$ of holomorphic functions at the origin
in $\Cn$, there is a canonical associated residue current, the
\emph{Coleff-Herrera product} $R^f_{CH}=\dbar[1/f_p]\w\cdots\w\dbar
[1/f_1]$, introduced in \cite{CH}. 
It has support on $\{f_1=\ldots = f_p=0\}$ and satisfies the
\emph{duality principle (\cite{DS,P}): A holomorphic function $\xi$ is locally
in the ideal $(f)$ generated by $f_1, \ldots, f_p$ if and only if
$\xi$ annihilates $R_{CH}^f$, i.e., 
$\xi R^f_{CH}=0$}. 
Given a free resolution of an ideal (sheaf) $\a$ of holomorphic
functions, Andersson and the second author constructed in \cite{AW} a
vector-valued residue
current $R$ that satisfies the duality principle and that coincides with $R^f_{CH}$
if $\a$ is a complete intersection ideal, generated by a regular sequence $f_1,\ldots, f_p$, see Section ~\ref{residue}.  
This construction has recently been used, e.g., to obtain new results
for the $\dbar$-equation and effective
solutions to polynomial ideal membership problems on singular
varieties, see, e.g., \cite{AS, AS2, ASS, AW4, Sz}.

In this paper we compute the current $R$ for the \emph{hull resolution} (and
more general cellular resolutions), introduced by Bayer-Sturmfels \cite{BaS}, of Artinian,
i.e., $0$-dimensional, monomial ideals, extending previous results by
the second author. 
The hull resolution of a monomial ideal $M$  
is encoded in
the \emph{hull complex} $\hull (M)$, which is a labeled polyhedral cell
complex in $\R^n$ of dimension $n-1$ with one vertex for each minimal
generator of $M$. The face $\sigma\in \hull (M)$ is labeled by the least common multiple of the monomials corresponding to the vertices of $\sigma$, 
see Section ~\ref{cell}.

\begin{thm}\label{main}
Let $M$ be an Artinian monomial ideal in $\C^n$ and let $R$ be the
residue current constructed from the hull resolution of $M$. Then $R$ has one entry
$R_\sigma$ for each $(n-1)$-dimensional face $\sigma$ of $\hull (M)$, and 
\begin{equation*}
R_\sigma=
\dbar\left [\frac{1}{z_n^{\alpha_n}}\right]\w\cdots\w 
\dbar\left [\frac{1}{z_1^{\alpha_1}}\right],
\end{equation*}
where $z_1^{\alpha_1}\cdots z_n^{\alpha_n}$ is the label of $\sigma$. 
\end{thm}

If $M$ is a complete intersection ideal, $\hull (M)$ is an $(n-1)$-simplex
and the hull resolution is the Koszul complex.
In general, $\hull (M)$ is a polyhedral subdivision of an $(n-1)$-simplex.
In fact, Theorem~\ref{main} holds for more general cellular resolutions,
where the underlying polyhedral cell complex is a polyhedral subdivision
of the $(n-1)$-simplex, see Theorem~\ref{main2}.

It was proved in \cite{CH} that if $f_1,\ldots, f_p$ is a regular sequence, then 
\begin{equation}\label{enstjarna}
R^f_{\text{CH}}\w \frac{df_1\w\cdots\w df_p}{(2\pi i)^p} = [(f)],
\end{equation} 
where $[(f)]$ is the fundamental cycle of the ideal $(f)$. 
Our main motivation to compute $R$ explicitly 
was to understand a similar factorization of the fundamental cycle of an
arbitrary ideal. By computing $d\varphi:=
d\varphi_0\circ\cdots\circ d\varphi_{n-1}$, where $\varphi_k$ are the
maps in the (hull) resolution of a (generic) Artinian monomial ideal
$\a$, and using Theorem \ref{main}, we get
\begin{equation}\label{tvastjarna}
\frac{d\varphi}{n! (2\pi i)^n}\circ R = [\a], 
\end{equation} 
see Section ~\ref{fundamental}. 
Since $\a$ is Artinian, $[\a]=m[0]$, where $m$ is the \emph{geometric
  multiplicity} $\dim_\C \Ok_0^n/\a$ of $\a$, see
\cite[Section~1.5]{F}. Moreover, since $\a$ is monomial, $m$ equals
the volume of the \emph{staircase} $\R^n_+\setminus
\bigcup_{z^\alpha\in \a} \{\alpha +
\R^n_+\}$ of $\a$. 
If $\a$ is a complete intersection ideal generated by $f_1,\ldots,
f_n$, then $d\varphi=n! df_1\w\cdots\w df_n$, and thus
\eqref{tvastjarna} can be seen as a generalization of
\eqref{enstjarna}. 
We recently managed to prove a generalized version of
\eqref{tvastjarna} for arbitrary ideals of pure dimension; this is a
current version of 
(a generalization of) a result due to Lejeune-Jalabert \cite{L-J} and will
be the subject of the forthcoming paper \cite{LW}.

In \cite{W3} the current $R$ was computed as the push-forward of a
certain 
current in a toric resolution of the ideal $M$. 
The main result in that paper asserts that each $R_\sigma$ is of the
form $R_\sigma= c_\sigma \dbar[1/z_n^{\alpha_n}]\w\cdots\w\dbar 
  [1/z_1^{\alpha_1}]$ for some $c_\sigma\in \C$. 
The coefficients $c_\sigma$ appear as integrals that
  seem to be hard to compute in general, see Section \ref{examples}. 
The proof of Theorem~\ref{main} given here is different and
more direct. A key tool is a comparison formula for residue
currents due to the first author. If 
\begin{equation*}\label{karv2}
0\to \mathcal O(E_{n-1})\stackrel{\varphi_{n-1}}{\longrightarrow} \cdots\stackrel{\varphi_1}{\longrightarrow} \mathcal O(E_0)\stackrel{\varphi_0}{\longrightarrow}\mathcal O(E_{-1})
\end{equation*}
is a resolution of an Artinian ideal $\a$ and $\ldots
\to \Ok (F_k) \to \Ok(F_{k-1}) 
\to \ldots$ is a resolution of $\b\subset \a$, then there are (locally) maps
$a_k: \Ok(F_k)\to \Ok(E_k)$, so that the diagram
\begin{equation*}\label{lediagram2} 
\xymatrix{
0 \ar[r] & \Ok (E_{n-1}) \ar[r]^{\varphi_{n-1}} & \dots \ar[r]^{\varphi_1} & \Ok(E_0) \ar[r]^{\varphi_0} & \Ok(E_{-1})\\
0 \ar[r] & \Ok(F_{n-1}) \ar[r]^{\psi_{n-1}} \ar[u]^{a_{n-1}} & \dots \ar[r]^{\psi_1}& \Ok(F_0) \ar[r]^{\psi_0} \ar[u]^{a_0} & \Ok(F_{-1}) \ar[u]^{a_{-1}}  
} 
\end{equation*} 
commutes. Theorem~1.3 in \cite{L} asserts that $R^E a_{-1}= a_{n-1} R^F$ if
$R^E$ and $R^F$ are the currents associated with $\Ok(E_\bullet)$ and
$\Ok(F_\bullet)$, respectively, see Section ~\ref{homotopy}.

The main ingredient in the proof of Theorem ~\ref{main} is Proposition
~\ref{ana}, which gives an explicit
description of mappings $a_k$ when $\Ok(E_\bullet)$ and
$\Ok(F_\bullet)$ are cellular resolutions such that the underlying
polyhedral cell complex of $\Ok(E_\bullet)$ refines the polyhedral
cell complex of $\Ok(F_\bullet)$, and which we have not managed to
find in the literature. 
Letting $\Ok(E_\bullet)$ be the hull resolution of $M$ and
$\Ok(F_\bullet)$ the Koszul complex of a sequence $z_1^{b_1}, \ldots,
z_n^{b_n}$ contained in $M$, so that $R^F$ is the simple
Coleff-Herrera product $\dbar[1/z_n^{b_n}]\w\cdots\w\dbar
[1/z_1^{b_1}]$, we can then easily compute $R^E$.

The paper is organized as follows. In Sections ~\ref{residue} and
\ref{cell} we provide some background on residue currents and cellular
resolutions, respectively. In Section ~\ref{komplex} we prove some
basic results concerning oriented polyhedral complexes, which are
needed for the proof of Theorem~\ref{main} (and the slightly more
general Theorem~\ref{main2}). The proof occupies Section
\ref{bevis}. In Section ~\ref{examples} we compare Theorems ~\ref{main}
and ~\ref{main2} to
previous results and also illustrate them by some examples. In
Section ~\ref{higher} we consider residue currents of non-Artinian monomial
ideals, and, finally, in Section ~\ref{fundamental} we discuss the relation to
fundamental cycles. 

\smallskip 
\noindent
\textbf{Acknowledgment.} 
We would like to thank Mats Andersson, Mattias Jonsson, and Mircea Musta\c{t}\u{a} for
helpful discussions. We would also like to thank the referee for valuable
comments and suggestions. The second author was supported by the Swedish Research Council.

\section{Residue currents}\label{residue}

Given a holomorphic function $f$ we will write $[1/f]$ (or sometimes just $1/f$) for the \emph{principal value distribution} of $1/f$, which can be realized, e.g., as the limit of the smooth approximands $\frac{\bar f}{|f|^2+\epsilon}$. 
If $f$ is a regular sequence of (germs of) holomorphic functions $f_1,\ldots, f_p$ one can give meaning to products of principal values $[1/f_j]$ and \emph{residue currents} $\dbar[1/f_k]$, 
as was first done in \cite{CH}, see also \cite{P2}. The products can be defined, e.g., by taking the limit of products of the
corresponding forms $\frac{\bar f_j}{|f_j|^2+\epsilon}$ and $\dbar\frac{\bar f_k}{|f_k|^2+\epsilon}$. 
They are (anti-)commutative in the factors and satisfy Leibniz' rule:
If $f_k=g_1\cdots g_s$, then 
\begin{equation}\label{leibnizregeln}
\dbar\left [\frac{1}{f_k}\right ]\w \cdots \w \dbar\left [\frac{1}{f_1}\right ]=
\sum_j \left [\frac{1}{g_1\cdots \hat g_j\cdots g_s}\right ]\dbar\left [\frac{1}{g_j}\right ]\w\dbar\left [\frac{1}{f_{k-1}}\right ]\w \cdots \w \dbar\left [\frac{1}{f_1}\right ].
\end{equation}
We will denote the \emph{Coleff-Herrera product} $\dbar [1/f_p]\w\cdots\w\dbar [1/f_1]$ of $f$ by $R_{CH}^f$. 
If $f_j=z_j^{b_j}$ for $j=1,\ldots, n$, then the action of $R_{CH}^f$
on the test form $\xi(z)dz_1\w\cdots\w dz_n$ equals 
\[
\frac{(2\pi i)^n}{(b_1-1)!\cdots (b_n-1)!}
\frac{\partial^{b_1+\cdots +b_n-n}}
{\partial z_1^{b_1-1} \cdots \partial z_n^{b_n-1}}
\xi (0).
\]

Consider a complex of Hermitian holomorphic vector bundles over a
complex manifold ~$X$ of dimension $n$, 
\begin{equation}\label{bunt}
0\to E_N\stackrel{\varphi_N}{\longrightarrow}\ldots\stackrel{\varphi_2}{\longrightarrow} E_1\stackrel{\varphi_1}{\longrightarrow}
E_0\stackrel{\varphi_0}{\longrightarrow}E_{-1},
\end{equation}
that is exact outside an analytic variety ~$Z\subset X$ of positive
codimension $p$. Suppose that the rank of ~$E_{-1}$ is $1$. 
In \cite{AW} Andersson and the second author constructed an $\End
(\bigoplus E_k)$-valued current $R=R^E$ that in a certain sense measures the
lack of exactness of the associated sheaf complex of holomorphic sections 
\begin{equation}\label{karv}
0\to \mathcal O(E_N)\stackrel{\varphi_N}{\longrightarrow} \cdots\stackrel{\varphi_1}{\longrightarrow} \mathcal O(E_0)\stackrel{\varphi_0}{\longrightarrow}\mathcal O(E_{-1}).
\end{equation}
The current $R$ has support on $Z$ and if $\xi\in\Ok(E_{-1})$ satisfies
$R\xi=0$ then $\xi\in \Im \varphi_0$. If \eqref{karv} is exact,
i.e., if it is a locally free resolution of the sheaf $\Ok(E_{-1})/\Im
\varphi_0$, then $R\xi =0$ if and only if $\xi\in\Im \varphi_0$. 
The grading in \eqref{bunt} is somewhat unorthodox; in \cite{AW} the
complex ends at $E_0$. In this paper the grading is shifted by one
step, in order to make it fit the grading of the hull complex better.

Let ~$R^\ell_k$ denote the component of ~$R$ that takes values in
$\Hom(E_{\ell-1}, E_{k-1})$ and let $R^\ell=\sum_k R^\ell_k$. 
The shifting of the indices here is motivated by the shifting of the grading
of \eqref{bunt} compared to \cite{AW}. 
If \eqref{karv}
is exact, then $R^\ell=0$ for $\ell\geq 1$. We then write $R_k=R_k^{0}$
without any risk of confusion. The current $R_k$ has bidegree $(0,k)$, and
thus, by the dimension
principle for residue currents (see \cite{AW2}, Corollary~2.4), $R_k=0$ for $k<p$, and 
for degree reasons, $R_k = 0$ for $k>n$. 
In particular, if \eqref{karv} is a resolution of length $p$ of a
\emph{Cohen-Macaulay} ideal sheaf, i.e., at each $x\in X$, there is a
resolution of length $p$ (so that \eqref{karv} ends at level $p-1$), then $R=R_p$.  
In this case, $R$ is independent of the Hermitian metrics on the bundles $E_k$. 
By Hilbert's syzygy theorem, each $0$-dimensional ideal sheaf is
Cohen-Macaulay.

The degree of explicitness of the current $R$ of course 
depends on the degree of explicitness of the complex 
\eqref{bunt}. In general it is hard to find explicit free resolutions. 
In Section ~\ref{cell} we will describe a method for constructing free
resolutions of monomial ideals due to Bayer-Sturmfels ~\cite{BaS}.

\begin{ex}\label{koszulex}
Let $f$ be a sequence of holomorphic functions $f_1,\ldots,
f_p$ in a domain $\Omega$ in $\C^n$, and let \eqref{bunt} be the Koszul complex of $f$: Identify
$f$ with a section $f=\sum f_j e_j$ of a trivial vector bundle $\widetilde E$ of
rank $p$ over $\Omega$ 
with frame $e_j$. 
Let $E_{k-1}$ be the $k$th exterior product $\Lambda^k \widetilde E^*$ of the dual bundle
$\widetilde E^*$, equipped with the trivial metric, 
and let $\varphi_{k-1}$ be
contraction $\delta_f$ with $f$, i.e., 
\begin{equation*}
\delta_f: e_{i_1}^*\w\cdots \w e_{i_k}^* \mapsto 
\sum_j (-1)^{j-1} f_{i_j}   e_{i_1}^*\w\cdots \wedge e_{i_{j-1}}^*  \w e_{i_{j+1}}^*\w\cdots \w e_{i_k}^*,
\end{equation*}
where $e_j^*$ is the dual frame to $e_j$.
Then the entries of $R^E$
are the \emph{Bochner-Martinelli residue currents of $f$} in the sense of
Passare-Tsikh-Yger \cite{PTY},
see \cite{A}. 
If $f$ defines a complete intersection ideal $\a$, 
then the Koszul complex of
$f$ is a resolution of 
$\a$ and 
the current $R^E=R^E_{p}$ then equals the Coleff-Herrera product 
$R_{CH}^f$ (times $e_1^*\w\cdots\w e_p^*$), see
\cite[Theorem~4.1]{PTY} or \cite[Theorem~1.7]{A}. 
The currents $R^E$ can thus be seen as generalizations of the
Coleff-Herrera products and the fact that $R^E\xi=0$ if and only if $\xi\in\Im
\varphi_0$ when \eqref{karv} is exact can be seen as an extension of the duality principle
for Coleff-Herrera products.

\end{ex}

\subsection{A comparison formula for residue currents}\label{homotopy} 

Assume that $E_\bullet, \varphi_\bullet$ and $F_\bullet, \psi_\bullet$
are Hermitian complexes of vector bundles and that there are
holomorphic mappings $a_k:\Ok(F_k)\to \Ok(E_k)$ so that the diagram 
\begin{equation}\label{lediagram} 
\xymatrix{
0 \ar[r] & \Ok (E_N) \ar[r]^{\varphi_N} & \dots \ar[r]^{\varphi_1} & \Ok(E_0) \ar[r]^{\varphi_0} & \Ok(E_{-1})\\
0 \ar[r] & \Ok(F_N) \ar[r]^{\psi_N} \ar[u]^{a_N} & \dots \ar[r]^{\psi_1}& \Ok(F_0) \ar[r]^{\psi_0} \ar[u]^{a_0} & \Ok(F_{-1}) \ar[u]^{a_{-1}}  
} 
\end{equation}
commutes. 
For example, if the sheaf complex \eqref{karv} is exact and $\Im \psi_0\subset \Im \varphi_0$ one
can always find maps $a_k:  \Ok_x(F_k)\to \Ok_x(E_k)$ for each $x\in
X$, so that the corresponding diagram commutes, see
\cite[Proposition~A3.13]{E}. 

In \cite{L} the residue currents associated
with $E_\bullet, \varphi_\bullet$ and $F_\bullet, \psi_\bullet$ are related in terms of the morphisms
$a_k$. Assume that $\Ok(E_\bullet), \varphi_\bullet$ and $\Ok(F_\bullet),
\psi_\bullet$ are locally free resolutions of minimal length of $\Ok(E_{-1})/\a$ and
$\Ok(F_{-1})/\b$, respectively, where $\a$ and $\b$ 
are Cohen-Macaulay ideals of codimension $p$. 
Then Theorem~1.3 in \cite{L}
asserts that 
\begin{equation}\label{tricket}
R^E a_{-1}=a_{p-1} R^F.
\end{equation}
We will apply \eqref{tricket} to the situation where $\a$ and $\b$ are
ideals of $\Ok(E_{-1})=\Ok(F_{-1})$ such that $\b\subset\a$
(and $a_{-1}$ is the isomorphism $\Ok(F_{-1})\cong\Ok(E_{-1})$).

If $E_\bullet,\varphi_\bullet$ and $F_\bullet,\psi_\bullet$ are Koszul complexes of regular
sequences $f_1,\ldots, f_p$ and $g_1,\ldots, g_p$, respectively, such that
$[g_p \ldots g_1]^T= A [f_p \ldots f_1]^T$ for some holomorphic matrix $A$, then \eqref{tricket} is just
the \emph{transformation law} for Coleff-Herrera products:
\begin{equation} \label{translaw}
    R^f_{CH}= \det(A) R^g_{CH},
\end{equation}
see \cite[Remark~2]{L}.

\section{Oriented polyhedral cell complexes}\label{komplex} 

Recall that a \emph{face} of a polytope $\sigma$ is the intersection
of $\sigma$ and a supporting hyperplane of $\sigma$.  
A \emph{polyhedral cell complex} ~$X$ is a finite collection of convex
polytopes in $\R^n$ for some $n$, the \emph{faces} of ~$X$, that satisfy that if
$\sigma\in X$ and ~$\tau$ is a face of ~$\sigma$, 
then $\tau\in X$,
and moreover if ~$\sigma$ and ~$\sigma'$ are in ~$X$, then $\sigma\cap
\sigma'$ is a face of both ~$\sigma$ and ~$\sigma'$. 
For a reference on polytopes and polyhedral cell complexes, see, e.g., \cite{Z}.
The dimension of
a face ~$\sigma$, $\dim \sigma$, is defined as the dimension of its
affine hull (in $\R^n$) and the dimension of ~$X$, $\dim X$, is
defined as $\max_{\sigma\in X} \dim \sigma$. 
Let ~$X_k$ denote the set
of faces of ~$X$ of dimension $k$; $X_{-1}$ should be interpreted as
$\{\emptyset\}$. 
If $\dim \sigma=k$, then a face of $\sigma$ of dimension $k-1$ is said to be a
\emph{facet} of $\sigma$. 
Faces of 
dimension ~$0$ are called \emph{vertices} and faces of dimension ~$1$
are called \emph{edges}.  A face ~$\sigma$ is a \emph{simplex} if the number of vertices
is equal to $\dim \sigma+1$. A polyhedral cell
complex $X'\subset X$ is said to be a \emph{subcomplex} of
$X$. 

We will write $|X|$ for the union of all faces in $X$. 
A \emph{polyhedral subdivision} of a
polytope $\sigma\subset\R^n$ is a polyhedral cell complex $X$, such that
$|X|=\sigma$. 
If $Y$ is a polyhedral cell complex such that $|X| = |Y|$ and  
each face in $Y$
is a union of faces in $X$; we say that $X$ \emph{refines} $Y$.

%The following lemma is probably well-known, but we include a proof for
%the reader's convenience. Note that the assumption that $|X|$ is
%convex is crucial. For example, the lemma fails to hold if $X$
%consists of three edges meeting at a single vertex. 

The following lemma can be proved by standard arguments, cf., e.g.,
\cite{Z}. Note that the assumption that $|X|$ is
convex is crucial. For example, the lemma fails to hold if $X$
consists of three edges meeting at a single vertex. 

\begin{lma}\label{fallen}
Let $X$ be a polyhedral cell complex of
dimension $k\geq 1$, such that $|X|$ is a convex polytope. Consider
$\tau\in X_{k-1}$. If $\tau$ is
contained in the boundary of $|X|$, there is a unique $\sigma\in
X_k$ such that $\tau$ is a facet of $\sigma$. Otherwise there are
precisely two faces $\sigma_1,\sigma_2\in X_k$ such that $\tau$ is a
facet of $\sigma_1$ and $\sigma_2$. 
\end{lma}

\subsection{Orientation}\label{oriental}
For a convex set $S\subset \R^n$ we let $\spann S$ be the underlying
vector space of the affine hull of $S$. 
In other words, $\spann S$ is the subspace of $\R^n$ generated by
vectors of the form $\rho_1-\rho_2$, where $\rho_1,\rho_2\in S$. By an \emph{oriented polytope} in $\R^n$ we will mean a polytope
$\sigma\subset \R^n$ together with an orientation of the subspace
$\spann\sigma$. 
Within this section will write $\underline\sigma$ for the polytope and reserve $\sigma$
for the oriented polytope. 
Recall that an orientation of $\spann \underline\sigma$
is determined by a linear form, which we denote by $\omega_\sigma$, on
$\Lambda^k (\spann\underline\sigma)$ if $\dim \underline\sigma=k\geq 1$;
a basis $w_1,\ldots, w_k$ of $\spann\underline\sigma$ is positively
oriented if and only if $\omega_\sigma(w_1\w\cdots\w w_k)>0$. 
There is only one way of orienting polytopes of dimension
$0$ as well as the empty set.

\begin{remark}
An oriented simplex can equivalently be seen as a simplex together with  an
equivalence class of the total ordering of the vertices, where two
orderings are equivalent if and only if they differ by an even
permutation. 
We write $[v_1,\ldots , v_{k+1}]$ for the simplex with
vertices $v_1, \ldots, v_{k+1}$ together with the equivalence class of the ordering
$v_1<\ldots <v_{k+1}$, and
$-[v_1,\ldots , v_{k+1}]$ for the simplex with the opposite orientation,
cf.\ for instance, \cite[Chap.~4]{Sp}. If $\underline\sigma$ is a simplex with vertices
$v_1,\ldots, v_{k+1}$, we identify $\sigma=[v_1,\ldots, v_{k+1}]$ with $\sigma$ oriented
so that the basis $v_1-v_{k+1},\cdots,v_k-v_{k+1}$ of $\spann \underline\sigma$ is positively oriented. 
\end{remark}

An oriented polytope $\sigma$ of dimension $k\geq 2$ induces orientations of the 
facets of $\sigma$ in the following way: 
Let $\underline\tau$ be a facet of $\underline\sigma$, and let $\eta$ be a normal vector
to the affine hull of $\underline\tau$ in the affine hull of $\underline\sigma$
pointing in the direction of $\underline\sigma$. We will say that such a vector $\eta$
is a \emph{normal vector to $\underline\tau$ pointing inwards to $\underline\sigma$}.
Then, the orientation of $\spann\underline\tau$ induced by $\sigma$ is
defined by that a basis $w_1,\ldots,w_{k-1}$ of $\spann \underline\tau$ is positively oriented
if and only if $\eta, w_1,\ldots,w_{k-1}$ is a positively oriented basis of $\spann \underline\sigma$.
If $\sigma$ is a simplex $[v_1,\ldots, v_{k+1}]$ and $\underline\tau$ is obtained from $\underline\sigma$ by removing the vertex $v_j$, then it is easily verified that $\sigma$ induces the orientation $(-1)^{j-1}[v_1,\ldots, v_{j-1},v_{j+1},\ldots, v_{k+1}]$ of $\tau$.

We say that a polyhedral cell complex is
\emph{oriented} if each face is equipped with an orientation. More
precisely, an oriented polyhedral cell complex is a finite collection of
oriented polytopes $\sigma$, such that the underlying polytopes
$\underline\sigma$ form a polyhedral cell complex; we say that $\tau$
is a face of $\sigma$ if $\underline\tau$ is a face of
$\underline\sigma$ etc.

If $X$ is an oriented polyhedral cell complex, $\sigma\in X_k$, and $\tau\in
X_{k-1}$ is a facet of $\sigma$, let $\sign(\tau,\sigma)=1$ if the
orientation of $\tau$ induced by the orientation of $\sigma$ coincides
with the orientation of $\tau$, and let $\sign(\tau,\sigma)=-1$
otherwise. 
If $w_1,\ldots,w_{k-1}$ is a basis of $\spann \underline\tau$, and $\eta$ is a
normal vector of $\underline\tau$ pointing inwards to $\underline\sigma$, then
\begin{equation} \label{eqsign1} 
    \sign(\tau,\sigma) = \sgn\big(\omega_\sigma(\eta\w w_1\w\cdots\w
    w_{k-1})\big)/\sgn\big(\omega_\tau(w_1\w \cdots\w w_{k-1})\big).
\end{equation}
If $k = 1$, 
we interpret $\sign(\tau,\sigma)$ as $1$ if the normal $\eta$ pointing
inwards to $\underline\sigma$ is positively oriented, and $-1$ otherwise, and if $k=0$ we interpret $\sign(\tau,\sigma)$ as $1$.  
This is consistent with \eqref{eqsign1} if we interpret $\omega_\sigma$
as $1$ if $\dim\sigma\leq 0$.

Similarly if $\sigma\in X_k$ and $\sigma'$ is any oriented polytope of dimension $k$ that is
contained in $\sigma$ (i.e., $\underline\sigma'\subset\underline\sigma$), let $\sign(\sigma',\sigma)=1$ if the
orientation of $\spann \underline \sigma'=\spann \underline \sigma$
given by $\sigma'$ coincides with the
orientation given by $\sigma$ and let $\sign(\sigma',\sigma)=-1$
otherwise.
If $w_1,\ldots,w_k$ is a basis of $\spann \underline\sigma$, then
\begin{equation} \label{eqsign2}
    \sign(\sigma',\sigma) = \sgn\big(\omega_\sigma(w_1\w \cdots\w
    w_k)\big)/\sgn\big(\omega_{\sigma'}(w_1\w \cdots\w w_k)\big).
\end{equation}
If $k\leq 0$, $\sign(\sigma',\sigma)$ should be interpreted as $1$.

\begin{lma}\label{forfining}
Let $X$ and $X'$ be oriented polyhedral cell complexes such that $X'$
refines $X$. Assume that $\sigma'\subset \sigma$, where $\sigma\in X_k$ and $\sigma'\in X'_k$. Moreover assume that $\tau\in X_{k-1}$ and
$\tau'\in X'_{k-1}$ are facets of $\sigma$ and $\sigma'$,
respectively, and that $\tau'\subset\tau$. Then 
\begin{equation}\label{storm}
\sign(\sigma',\sigma)\sign(\tau',\sigma')=
\sign(\tau,\sigma)\sign(\tau',\tau).
\end{equation} 
\end{lma}

\begin{proof}
    Let $\eta$ be a normal vector of $\underline\tau'$ pointing inwards to $\underline\sigma'$.
    Then, $\eta$ is also a normal vector of $\underline\tau$ pointing inwards to $\underline\sigma$.
    Let $w_1,\ldots,w_{k-1}$ be a basis of $\spann \underline\tau' = \spann \underline\tau$. Then by \eqref{eqsign1} and \eqref{eqsign2},
    both 
    sides of \eqref{storm} are equal to
    \begin{equation*}
        \sgn\big (\omega_\sigma(\eta \w w_1 \w \cdots \w
        w_{k-1})\big)/\sgn\big(\omega_{\tau'}(w_1 \w\cdots\w w_{k-1})\big).
    \end{equation*}
\end{proof}

\begin{figure}
\begin{center}
\includegraphics{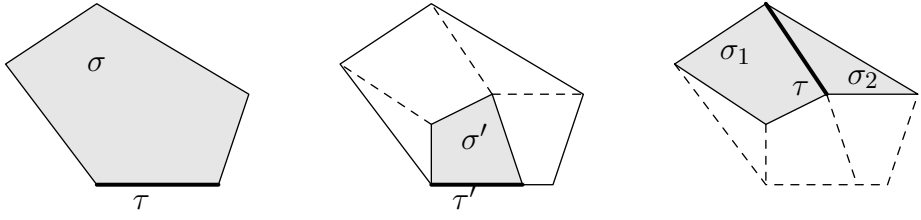}
\caption{Examples of faces $\sigma$, $\tau$, $\sigma'$, and $\tau'$ in
  Lemma ~\ref{forfining} (in the left and middle figure) and faces
  $\sigma_1$, $\sigma_2$ and $\tau$ in Lemma ~\ref{cancel} (in the
  right figure).}
\label{lemmorna}
\end{center}
\end{figure}

\begin{lma}\label{cancel}
Let $\sigma$ be an oriented polytope of dimension $k \geq 1$, and let $X$ be a polyhedral
subdivision of $\sigma$.
Assume that $\tau\in X_{k-1}$ is a facet of two faces $\sigma_1, \sigma_2\in X_k$. 
Then 
\begin{equation}\label{streckkod}
\sign(\sigma_1, \sigma)\sign(\tau, \sigma_1)+
\sign(\sigma_2, \sigma)\sign(\tau, \sigma_2)=0.
\end{equation} 
\end{lma}

\begin{proof}
    Being in the same situation as in the second case in Lemma~\ref{fallen},
    it is easily verified that we may assume that $|X|=\underline\sigma\subset \R^k_{x_1,\ldots, x_k}$,
    $\underline\tau\subset \{x_k=0\}$, and $\underline\sigma_j\subset H_j, j=1,2$,
    where $H_1 = \{ x_k \geq 0 \}$ and $H_2 = \{ x_k \leq 0 \}$.
Then the vector $\eta:=(0,\ldots, 0,1)$ is a normal vector to $\underline\tau$ pointing inwards to $\underline\sigma_1$ and $-\eta$ is a  normal vector to $\underline\tau$ pointing inwards to $\underline\sigma_2$. Letting $w_1,\ldots,w_{k-1}$ be a basis of $\spann \underline\tau$, 
   by \eqref{eqsign1} and \eqref{eqsign2}  the first term in the left-hand side of \eqref{streckkod} equals
    \begin{equation} \label{eqsgnsigmatau}
        \sgn\big (\omega_\sigma(\eta\w w_1 \w \cdots \w w_{k-1})\big
        )/\sgn\big (\omega_\tau(w_1 \w\cdots \w w_{k-1})\big )
    \end{equation}
    and the second term equals \eqref{eqsgnsigmatau} with the opposite sign. 

\end{proof}

\section{Cellular resolutions of monomial ideals}\label{cell}

Let us recall the construction of cellular resolutions due to
Bayer-Sturmfels \cite{BaS}. 
Let ~$S$ be the polynomial ring $\C[z_1,\ldots,z_n]$. 
We say that an (oriented) polyhedral cell complex ~$X$ is \emph{labeled} if there
is a monomial 
$m_i$ in ~$S$ associated with each vertex $v_i$. An arbitrary face
~$\sigma$ of ~$X$ is then labeled by the least common multiple of the
labels of the vertices of ~$\sigma$, i.e., by 
$m_\sigma=\text{lcm}\{m_{i}|i\in \sigma\}$; $m_\emptyset$ should be
interpreted as $1$. 
We will sometimes be sloppy and not differ between the faces of a labeled complex and their labels.

\begin{df}
If $X$ and $Y$ are two labeled polyhedral cell complexes, we say that $X$ \emph{refines}
$Y$ if $X$ refines $Y$ as polyhedral cell complexes, i.e., $|X| = |Y|$, and each
face of $Y$ is a union of faces in $X$, and in addition, we require that
if $\sigma' \in X$, $\sigma \in Y$, and $\sigma' \subset \sigma$, then
$m_{\sigma'} | m_{\sigma}$.
Note that this implies that the ideal generated by the labels of the vertices of $Y$
must be contained in the ideal generated by the labels of the vertices of $X$.
\end{df}

Let ~$M$ be a \emph{monomial ideal} in ~$S$, i.e., $M$ can be generated by
monomials. We will use the shorthand notation $z^\alpha$ for the monomial
$z_1^{\alpha_1}\cdots z_n^{\alpha_n}$ in $S$. 
It is easy to check that a monomial ideal has a
unique minimal set of generators that are monomials; assume that $\{m_1,
\ldots, m_r\}$ is a minimal set of monomial generators of $M$. 
Next, let ~$X$ be an oriented polyhedral cell complex with vertices $\{1,\ldots,r\}$
labeled by ~$\{m_1,\ldots, m_r\}$.
We will associate with ~$X$ a graded complex of free $S$-modules: For
$k=-1,\ldots, \dim X$, let ~$A_k$ be the free $S$-module with basis $\{e_\sigma\}_{\sigma\in \xk}$ 
and let the differential $\varphi_k:A_k \to A_{k-1} $ be defined by 
\begin{equation}\label{sabel}
\varphi_k: e_\sigma  \mapsto  
\sum_{\text{facets }\tau\subset \sigma} \sign(\tau,\sigma)~\frac{m_\sigma}{m_{\tau}}~ e_{\tau}.
\end{equation}
Note that $m_\sigma/m_{\tau}$ is a monomial when $\tau$ is a face of $\sigma$. The complex
\[
\F_X:0\to A_{\dim X}\stackrel{\varphi_{\dim
    X}}{\longrightarrow} \cdots \stackrel{\varphi_1}{\longrightarrow}
A_0\stackrel{\varphi_0}{\longrightarrow} A_{-1}
\]
is the \emph{cellular complex} \emph{supported on} ~$X$. 
Note that, with the identification $A_{-1}=S$, the cokernel of $\varphi_0$ equals $S/M$. 
The complex $\F_X$ 
is exact if the labeled complex ~$X$ satisfies a
certain acyclicity condition. More precisely, for $\beta\in \N^n$,
where $\N=\{0,1,\ldots\}$, let $X_{\preceq \beta}$ denote the subcomplex of ~$X$ consisting
of all faces ~$\sigma$ for which $z^\beta$ is divisible by
$m_\sigma$. 
Then ~$\F_X$ is
exact if and only if ~$X_{\preceq \beta}$ is acyclic, which means that
it is empty or has zero reduced homology, for all $\beta\in\N^n$, see
\cite[Proposition~4.5]{MS}. Note, in particular, that
the acyclicity does not depend on the orientation of $X$. 
When $\F_X$ is exact we say that it is a \emph{cellular resolution} of ~$S/M$.

To put the cellular resolutions into the context of ~\cite{AW},
let us consider the vector bundle complex ~\eqref{bunt},
where ~$E_k$ for $k=-1,\ldots, N=\dim X$ is a trivial bundle over
~$\C^n$ of rank equal to the number of faces in $X_k$,  
with a global frame $\{e_\sigma\}_{\sigma\in \xk}$,
endowed with the trivial metric, and where the
differential $\varphi_k$ is given by ~\eqref{sabel}. 
We will say that the corresponding residue current ~$R$ is associated
with ~$X$ and denote it by $R^X$, and we will use ~$R_\sigma$ to denote
the coefficient of $e_\sigma\otimes e_\emptyset^*$. 
The induced sheaf complex ~\eqref{karv} is exact if and only if ~$\F_X$ is. 
This follows from the standard fact that the ring $\Ok_0$ of germs of
holomorphic functions at $0\in\C^n$ is flat over $S$,
see for example \cite[Theorem~13.3.5]{Ta}.
We will think of monomial ideals sometimes as ideals in the polynomial
ring $S$, sometimes as ideals in the ring of entire functions in
$\C^n$, and sometimes as ideals in the local ring $\Ok_0^n$.

\subsection{The hull resolution}\label{hull}

Given a monomial ideal $M$ in $S$ and $t\in\R$, 
let $\mathcal P_t=\mathcal P_t(M)$ 
be the convex hull in $\R^n$ of 
$\{(t^{\alpha_1},\ldots, t^{\alpha_n})=:t^\alpha \mid z^{\alpha}\in M\}$. 
Then $\mathcal P_t$ is a unbounded 
polyhedron in $\R^n$ of dimension $n$ and the face poset (i.e., the set of faces
partially ordered by inclusion) of bounded faces of $\mathcal P_t$ is independent of $t$ if
$t\gg 0$. 
The \emph{hull complex} $\hull(M)$ of $M$, introduced in \cite{BaS}, is the polyhedral cell
complex of all bounded faces of $\mathcal P_t$ for $t\gg 0$. The vertices of
$\hull (M)$ are precisely the points $t^\alpha$, where $z^\alpha$ is a
minimal generator of $M$, and thus $\hull(M)$ admits a natural labeling. 
The corresponding complex $\F_{\hull (M)}$ is a resolution of
$S/M$; it is called the \emph{hull resolution}. 

\begin{ex}\label{sporty} 
Let $N$ be the complete intersection ideal $(z_1^{b_1}, \ldots, z_n^{b_n})$.
Then, $\hull(N)$ is the polyhedral cell complex consisting of the $(n-1)$-simplex
$\Delta = [v_1,\ldots,v_n]$ in $\R^n$ and its faces, where
$v_1=(t^{b_1},1,\ldots, 1), v_2=(1,t^{b_2},1,\ldots, 1), \ldots, v_n=(1,\ldots, 1,t^{b_n})$. 
The vertices $v_1,\ldots,v_n$ of $\hull(N)$ are labeled by
$z_1^{b_1},\ldots, z_n^{b_n}$, respectively, 
and we assume the faces are oriented so that the simplex $\sigma$ with vertices
$v_{i_1}, \ldots, v_{i_\ell}$ equals $[v_{i_1},\ldots, v_{i_\ell}]$ if ${i_1} < \ldots < {i_\ell}$. 
Then the corresponding cellular complex $\F_{\hull(N)}$ is the Koszul complex
of $(z_1^{b_1},\ldots,z_n^{b_n})$, and 
\begin{equation}\label{svart}
R^{\hull(N)}=\dbar \left [\frac{1}{z_n^{b_n}}\right ]\wedge\cdots \wedge 
\dbar \left [\frac{1}{z_1^{b_1}}\right ] e_\Delta \otimes e^*_\emptyset,
\end{equation} 
cf.\ Section ~\ref{oriental} and Example ~\ref{koszulex}. Note that a different orientation
of the top-dimensional  simplex $\Delta=[v_1,\ldots, v_n]$ would permute the residue factors
in \eqref{svart}.
\end{ex} 

The example shows that the hull complex of the complete intersection ideal
 is the cellular complex consisting of an $(n-1)$-simplex together with its
faces. In general, if $M$ is Artinian, $\hull(M)$ is a polyhedral subdivision of such an
$(n-1)$-simplex or, rather, it can be embedded as one, see, e.g., (the
proof of) Theorem~4.31 in \cite{MS}. 
We will need the following more precise description of this embedding.
To begin with, we note that an Artinian monomial ideal has monomials
of the form $z_1^{\beta_1},\ldots,z_n^{\beta_n}$ among its minimal monomial generators.
Note also that every other minimal generator has degree smaller than $\beta_i$ in ~$z_i$.

\begin{prop} \label{hullembedding}
    Let $M$ be an Artinian monomial ideal with $(z_1^{b_1},\ldots,z_n^{b_n})$
    among its minimal monomial generators. Let $N$ be the complete intersection ideal
    $(z_1^{b_1},\ldots,z_n^{b_n})$. Then $\hull(M)$ can be embedded as a refinement
    of $\hull(N)$ as labeled polyhedral cell complexes.
\end{prop}

We will be sloppy and not always distinguish between the hull complex of $M$ and
its embedding.

\begin{proof}
    That $\hull(M)$ refines $\hull(N)$ as polyhedral cell complexes
    is Theorem~4.31 in \cite{MS}. In fact, it follows from the proof 
    in \cite{MS} of that theorem that
    it is a refinement also as \emph{labeled} polyhedral cell complexes.
    To see this, we begin by recalling (slightly differently described) the construction
    of the embedding in that proof. 

    We know from Example~\ref{sporty} that $\hull(N)$ consists of the faces
    of the simplex $\Delta$ with vertices 
    $v_1=(t^{b_1}, 1, \ldots, 1), \ldots, v_n=(1,\ldots, 1,t^{b_n})$. 
    For a point $p \neq {\bf 1} := (1,\ldots,1)$, with $p_i \geq 1$, consider the line $\ell$
    through ${\bf 1}$ and $p$. Since $p_i \geq 1$, $\ell$ intersects $\Delta$ in a unique
    point, which we denote $\pi(p)$. Moreover, since $|\hull(M)|$ is contained in the set where $p_i \geq 1$,
    we get a map $\pi : |\hull(M)| \to \Delta$, which induces an embedding
    of $\hull(M)$ into $\Delta$ by letting the faces of the embedded complex be the
    images $\pi(\sigma)$, where $\sigma \in \hull(M)$ (with the same labeling).

    Consider a face $\sigma'$ of $\hull(M)$ such that
    $\pi(\sigma') \subseteq \sigma = [v_{i_1},\ldots,v_{i_k}]$. Then the vertices of
    $\pi(\sigma')$ must be contained in the set $\{ x \in \R^n \mid x_i = 1, i \neq i_1,\ldots,i_k \}$,
    since the $v_{i_j}$ are. 
    A vertex $v$ of $\hull(M)$ with label $m_v = z^\alpha$ has coordinates
    $(t^{\alpha_1},\ldots,t^{\alpha_n})$, so if $\pi(v)$ is contained in $\{ x_i = 1 \}$,
    then we must have $\alpha_i = 0$ in $m_v$. 
    It follows that $m_{\sigma'}$ is of the form
    $m_{\sigma'} = z_{i_1}^{\alpha_{i_1}}\ldots z_{i_k}^{\alpha_{i_k}}$,
    and since each label of a minimal monomial generator is of degree at most $b_i$
    in $z_i$, the same must hold for $m_{\sigma'}$ since it is the common multiple of such labels.
    Hence, $m_{\sigma'} | m_\sigma = z_{i_1}^{b_{i_1}}\ldots z_{i_k}^{b_{i_k}}$.
\end{proof}

Recall that a graded free resolution $A_\bullet, \varphi_\bullet$ 
is \emph{minimal} if and only if for each ~$k$, ~$\varphi_k$ maps a basis
of ~$A_k$ to a minimal set of generators of $\Im \varphi_k$, see,
e.g., \cite[Corollary~1.5]{E2}. 
The hull resolution is not minimal in general, cf.
Example ~\ref{maxkvadrat}. However, if $M$ is a \emph{generic} monomial
ideal in the sense of \cite{BPS, MSY}, 
the hull complex is simplicial, i.e., all faces are simplices, and it coincides with the 
\emph{Scarf complex} of $M$, which is a minimal
resolution of $S/M$, see \cite{BPS}. 
 The ideal $M$ is generic if whenever two distinct minimal
 generators ~$m_i$ and 
~$m_j$ have the same positive degree in some variable, then there
exists a third generator ~$m_k$ that strictly divides the least common
multiple $z^\alpha$ of ~$m_i$ and ~$m_j$, meaning that $m_k$ divides
$z_1^{\alpha_1-1}\cdots z_n^{\alpha_n-1}$. Note that when $n\leq 2$ all monomial ideals are generic. The Scarf complex of $M$ 
is the collection of subsets
$I\subset\{1,\ldots,r\}$ whose corresponding least common multiple
~$m_I:=\lcm_{i\in I} m_i$ is unique.

\section{Proof of Theorem~\ref{main}}\label{bevis}
We will prove a slightly more general version of Theorem~\ref{main}. 
If $N$ is a complete intersection ideal $(z_1^{b_1},\ldots,z_n^{b_n})$,
by Example~\ref{sporty}, $\hull(N)$ is the polyhedral cell complex consisting of the faces
of an oriented $(n-1)$-simplex $\Delta$, with vertices labeled by $z_1^{b_1},\ldots, z_n^{b_n}$.
In particular, $\hull(N)_{n-1}$ consists of only the simplex $\Delta$.

\begin{thm}\label{main2}
Let $M$ be an Artinian monomial ideal in $S=\C[z_1,\ldots, z_n]$. 
Assume that $\F_X$ is a cellular resolution of $S/M$ such that the underlying labeled polyhedral
cell complex $X$ refines the hull complex of a complete intersection ideal
$N = (z_1^{b_1},\ldots,z_n^{b_n})$, i.e., the $(n-1)$-simplex $\Delta$ with vertices labeled by
$z_1^{b_1},\ldots,z_n^{b_n}$.
Then the associated residue current $R^X$ has one entry
$R_\sigma$ for each $(n-1)$-dimensional face $\sigma$ of $X$, and 
\begin{equation*}
R_\sigma=\sign(\sigma,\Delta) 
\dbar\left [\frac{1}{z_n^{\alpha_n}}\right]\w\cdots\w 
\dbar\left [\frac{1}{z_1^{\alpha_1}}\right],
\end{equation*}
where $z_1^{\alpha_1}\cdots z_n^{\alpha_n}$ is the label of $\sigma$.
\end{thm}
Theorem~\ref{main} corresponds to the case when $X$ equals $\hull (M)$; 
the refinement is given by Proposition~\ref{hullembedding}, and the orientation
of $\hull(M)$ is implicitly assumed to be such that $\sgn(\sigma, \Delta)=1$ for
each $\sigma\in \hull (M)_{n-1}$.

\begin{prop}\label{ana}
Let $X$ and $Y$ be oriented labeled polyhedral cell complexes such that
$X$ refines $Y$, and let $E_\bullet,\varphi_\bullet$ and $F_\bullet,\psi_\bullet$
be the corresponding vector bundle complexes. 
For $k\geq -1$ let $a_k:F_k\to E_k$ be the mapping 
\begin{equation}\label{gronprick}
a_k: e_\sigma\mapsto \sum_{\sigma'\subset \sigma}\sign(\sigma',\sigma)
\frac{m_\sigma}{m_{\sigma'}} e_{\sigma'},
\end{equation}
where the sum is over all $\sigma'\in X_{k}$ that satisfy 
$\sigma'\subset \sigma \in Y_k$. 
Then the $a_k$ are holomorphic and the diagram \eqref{lediagram} commutes.
\end{prop}

We let $X$ and $N$ be as in Theorem~\ref{main2}, and $Y = \hull(N)$.
Since $\dim X= \dim Y = n-1$, the complexes $E_\bullet,\varphi_\bullet$ and
$F_\bullet,\psi_\bullet$ end at level $n-1$. 
Thus, identifying $E_{-1}$ and $F_{-1}$ and taking Proposition~\ref{ana} for granted,
\eqref{tricket} yields  
\[
R^X=R^E=a_{n-1} R^F = 
\sum_{\sigma\subset\Delta} \sign (\sigma,\Delta) \frac{m_\Delta}{m_\sigma} 
\dbar \left [\frac{1}{z_n^{b_n}}\right ]\w\cdots\w 
\dbar \left [\frac{1}{z_1^{b_1}}\right ] e_\sigma\otimes
e^*_\emptyset; 
\]
here we have used \eqref{svart} for the last equality. Since
$|X|=|Y| = \Delta$, the sum is over all $\sigma\in X_k$, and 
since $m_\Delta=z_1^{b_1}\cdots z_n^{b_n}$ the coefficient of
$e_\sigma\otimes e^*_\emptyset$ is
just 
\[
\sign (\sigma,\Delta) 
\dbar \left [\frac{1}{z_n^{\alpha_n}}\right ]\w\cdots\w 
\dbar \left [\frac{1}{z_1^{\alpha_1}}\right ] , 
\]
where $z_1^{\alpha_1}\cdots z_n^{\alpha_n}=m_\sigma$. This concludes
the proof of Theorem~\ref{main2}.

\begin{proof}[Proof of Proposition ~\ref{ana}]
Since $X$ refines $Y$ as a labeled polyhedral cell complex, each $m_\sigma/m_{\sigma'}$ in \eqref{gronprick} is holomorphic and thus the $a_k$ are holomorphic. 

To show that \eqref{lediagram} commutes, we first consider the case $k\geq 1$. 
Pick $\sigma\in Y_k$. 
Then 
\begin{equation}\label{afsigma}
e_\sigma\stackrel{\psi_k}{\longmapsto}
\sum_{\tau\subset\sigma}\sign(\tau,\sigma)\frac{m_\sigma}{m_\tau}e_\tau
\stackrel{a_{k-1}}{\longmapsto} 
\sum_{\tau\subset\sigma}\sum_{\tau'\subset\tau}
\sign(\tau,\sigma)\sign(\tau',\tau)\frac{m_\sigma}{m_{\tau'}}e_{\tau'}.
\end{equation}
Here the first sum is over the facets $\tau\in Y_{k-1}$ of
$\sigma$, and the second sum is over the faces $\tau'\in X_{k-1}$ that are
contained in $\tau$. 
Moreover 
\begin{equation}\label{fasigma}
e_\sigma\stackrel{a_k}{\longmapsto}
\sum_{\sigma'\subset\sigma}\sign(\sigma',\sigma)\frac{m_\sigma}{m_{\sigma'}}e_{\sigma'}
\stackrel{\varphi_{k}}{\longmapsto} 
\sum_{\sigma'\subset\sigma}\sum_{\tau'\subset\sigma'}
\sign(\sigma',\sigma)\sign(\tau',\sigma')\frac{m_\sigma}{m_{\tau'}}e_{\tau'}.
\end{equation}
Now the first sum is over the faces $\sigma'\in X_k$ that are contained in
$\sigma$, whereas the second sum is over the facets $\tau'\in X_{k-1}$ of
$\sigma'$.

Let $X^\sigma$ be the $k$-dimensional subcomplex 
of faces of $X$ that are contained in $\sigma$ and consider $\tau'\in
X^\sigma_{k-1}$. Note that $X$ being a refinement of $Y$ means that $X^\sigma$
is a polyhedral subdivision of $\sigma$.
Assume that $\tau'$ is contained in a facet $\tau$ of
$\sigma$. Since $\dim \tau'=k-1=\dim\tau$, there is a unique such
$\tau$, and thus the coefficient of $e_{\tau'}$ (in the rightmost
expression) in \eqref{afsigma}
equals  
$\sign(\tau,\sigma)\sign(\tau',\tau)\frac{m_\sigma}{m_{\tau'}}$. 
Moreover, $\tau'$ is contained in the boundary of $|X^\sigma|$ and
thus by Lemma ~\ref{fallen} there is a unique $\sigma'\in
X^\sigma_k$ such that $\tau'\subset\sigma'$. Therefore the coefficient of $e_{\tau'}$ (in the rightmost
expression) in \eqref{fasigma} is 
$\sign(\sigma',\sigma)\sign(\tau',\sigma')\frac{m_\sigma}{m_{\tau'}}$. 
By Lemma ~\ref{forfining} these coefficients coincide.

If $\tau'$ is not contained in any facet $\tau$ of $\sigma$, then
clearly the coefficient 
of $e_{\tau'}$ in \eqref{afsigma} 
is zero. Also, then $\tau'$ is not contained in the boundary of $X^\sigma$,
and thus 
by Lemma ~\ref{fallen}, $\tau'$ is a facet of exactly two faces $\sigma'_1,
\sigma'_2\in X^\sigma_k$. Hence the coefficient of $e_{\tau'}$ in \eqref{fasigma} is 
\[
\big ( \sign(\sigma'_1,\sigma)\sign(\tau',\sigma'_1)+
\sign(\sigma'_2,\sigma)\sign(\tau',\sigma'_2) \big )
\frac{m_\sigma}{m_{\tau'}}, 
\]
which by Lemma ~\ref{cancel} vanishes. 
Since the sums in \eqref{afsigma} and
\eqref{fasigma} are only over $\tau',\sigma'\in X$ that are in
$X^\sigma$, it follows that $a_{k-1}\circ \psi_k (e_\sigma)=\varphi_{k}\circ a_k
(e_\sigma)$.

For $k=0$, pick a vertex $\sigma\in Y_0$. Since $X$ is a polyhedral subdivision of 
$Y$ and $\sigma$ is a vertex, the only $\sigma' \in X_0$ with $\sigma' \subset \sigma$
is $\sigma' = \sigma$. 
Thus $\varphi_0\circ a_0(e_\sigma)=
\varphi_0(m_\sigma/m_{\sigma'} e_{\sigma'}) = m_\sigma e_\emptyset$. Note that $a_{-1}$ maps
$e_\emptyset$ to $e_\emptyset$. Thus $a_{-1}\circ \psi_0 (e_\sigma)= m_\sigma
e_\emptyset$.

We conclude that $a_{k-1}\circ \psi_k =\varphi_{k}\circ a_k$ for
$k\geq 0$; in other words, the diagram \eqref{lediagram} commutes. 
\end{proof}

\section{Comparison to previous results}\label{examples}

In \cite{W3} the current $R=R^X$ constructed from a cellular resolution
$\F_X$ of an Artinian monomial ideal $M$ was computed up to
multiplicative constants; Proposition 3.1 in \cite{W3} asserts that $R$ has one entry
$R_\sigma$ 
for each face $\sigma\in X_{n-1}$, which is of the form 
\begin{equation}\label{gamla}
R_\sigma=c_\sigma 
\dbar\left [\frac{1}{z_n^{\alpha_n}}\right]\w\cdots\w 
\dbar\left [\frac{1}{z_1^{\alpha_1}}\right]
\end{equation}
for some $c_\sigma\in\C$, where $z_1^{\alpha_1}\cdots z_n^{\alpha_n}$
is the label of $\sigma$. The main novelty in this paper, except for the
new proof, is that we show that $c_\sigma=1$ (or $-1$, depending on
the orientation of $X$) and thus give a complete
description of ~$R$.

Let $\ann R\subset\Ok^n_0$ denote the \emph{annihilator ideal} of $R$,
i.e., the ideal of germs of holomorphic functions $\xi$ at $0\in\C^n$ that satisfy $R\xi=0$. 
Note that $\ann R_\sigma=(z_1^{\alpha_1}, \ldots,
z_n^{\alpha_n})=:\m^\alpha$. A monomial ideal of this form is said to be
\emph{irreducible}. Each monomial ideal $M$ can be
written as finite intersection of irreducible ideals; this is called an
\emph{irreducible decomposition} of $M$.
Since one has to annihilate each $R_\sigma$ in order to annihilate
$R$, Theorem~\ref{main2} implies that, provided $X$ is a polyhedral
subdivision of $\Delta$,
\[
\ann R=\bigcap_{\sigma\in X_{n-1}} \m^{\alpha_\sigma}, 
\]
which gives an irreducible decomposition of $\ann R=M$. Here
$\alpha_\sigma$ is the multidegree of the label of $\sigma$. 
If $\F_X$ is a minimal resolution of $M$ this decomposition is \emph{irredundant} in the sense that
no intersectand can be omitted. Each monomial ideal has a unique (monomial) irredundant irreducible decomposition.

Using that $R$ satisfies the duality
principle and results \cite[Theorem~3.7]{BPS} and
\cite[Theorem~5.42]{MS} about irreducible decompositions, in \cite{W3}, we could in some cases determine which
$c_\sigma$ are 
nonzero. If $M$ is a generic monomial ideal, Theorem~3.3 in that
paper says that $c_\sigma$ is
nonzero if and only if $\sigma$ is in the Scarf complex $\Delta_M$
(which is a subcomplex of any cellular resolution of $M$),
and if $\F_X$ is a minimal resolution of $M$ each $c_\sigma$ is
nonzero by Theorem~3.5 in \cite{W3}. 
Let us look at an example where 
these theorems do not apply.

\begin{ex}\label{maxkvadrat}
Consider the ideal 
$M=(z_1^2, z_1z_2, z_1z_3, z_2^2, z_2z_3, z_3^2)\subset S=\C[z_1,z_2,z_3]$, 
i.e., the square of the maximal ideal at $0$ in $S$. 
The hull complex of $M$ is a refinement 
of the $2$-simplex $\Delta$ with the vertices labeled by $z_1^2, z_2^2, z_3^2$, see Figure ~\ref{kvadratbild}. 

\begin{figure}
\begin{center}
\includegraphics{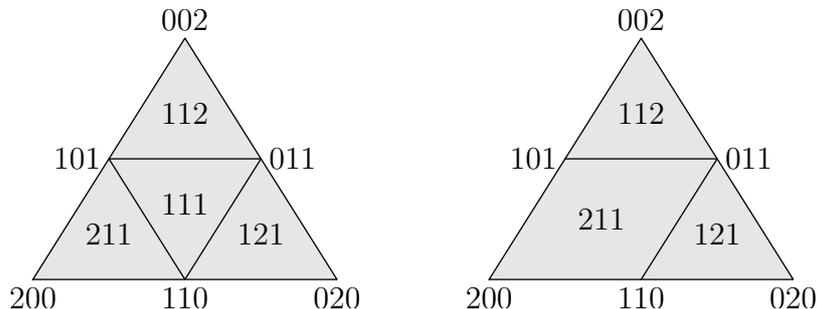}
\caption{The hull complex of the ideal $M$ in Example ~\ref{maxkvadrat}
  (labels on vertices and $2$-faces) (left) and the cell complex of a minimal free
  resolution of $M$ (right).}
  \label{kvadratbild}
\end{center}
\end{figure}

There are four faces $\sigma_1,\ldots, \sigma_4$ in $\hull_2(M)$ with
labels $m_{\sigma_1}=z_1^2z_2z_3$, $m_{\sigma_2}=z_1z_2^2z_3$, 
$m_{\sigma_3}=z_1z_2z_3^3$, and $m_{\sigma_4}=z_1z_2z_3$.  
By Theorem ~\ref{main2}, the current $R$ therefore has four entries: three entries of the
form 
$R_{\sigma_\ell}=\pm \dbar[1/z_k^2]\w\dbar[1/z_j]\w\dbar[1/z_i]$
for $\ell=1,2,3$, corresponding to the three corner triangles in $\hull(M)$, and one
component $R_{\sigma_4}=\dbar[1/z_3]\w\dbar[1/z_2]\w\dbar[1/z_1]$.

The hull resolution is not a minimal resolution of $S/M$. In
particular, 
$M$ is not generic. 
By arguing as in the proofs of Theorems 3.3 and 3.5 in \cite{W3}, using that
$R$ satisfies the duality principle and that 
$M=(z_1^2, z_2, z_3)\cap(z_1, z_2^2, z_3)\cap (z_1, z_2, z_3^2)$ 
is the irredundant irreducible decomposition of $M$, 
one can conclude that first three $c_{\sigma_j}$ in \eqref{gamla} are non-zero, 
but not that $c_{\sigma_4}$ is.

A minimal resolution of $S/M$ is obtained by removing one of the edges 
of the inner triangle in $\hull (M)$, see,
e.g., \cite[Example~3.19]{MS}. The cell complex $X$ of one such resolution is depicted in
Figure~\ref{kvadratbild}. 
Note that $X$ is a refinement of $\Delta$ 
(although different from $\hull(M)$) 
so that Theorem~\ref{main2} applies; the corresponding residue current consists of the three
entries $R_{\sigma_1}$, $R_{\sigma_2}$, and $R_{\sigma_3}$ above. 
\end{ex}

In \cite{W3} the current $R$ is computed as the push-forward of a
current on a toric log-resolution of $M$. The computations are inspired
by \cite{W}, where Bochner-Martinelli residue currents, cf. Example ~\ref{koszulex}, of monomial
ideals are computed, and they become quite involved. The coefficients
$c_\sigma$ appear as certain integrals in the log-resolution and seem
to be hard to
compute in general. 
The proof of Theorem~\ref{main2} given here is more direct and much less technical than in \cite{W3}.

It would be interesting to investigate whether the comparison formula
for residue currents could be used also to compute Bochner-Martinelli
residue currents. In \cite{W} it was shown that if $M$ is an Artinian
monomial 
ideal, the Bochner-Martinelli current $R_{BM}^M$ of (a monomial sequence of
generators of) $M$ is a vector-valued
current with entries of the form \eqref{gamla}, for certain exponents
$\alpha$. In some cases we can compute the 
coefficients $c_\sigma$, e.g., if $n=2$ and each minimal generator
of the monomial ideal $M$ is a vertex of the so-called Newton polytope of $M$; the coefficients
are then equal to $\pm 1$, see \cite[Section~4.2]{W4}. 

If $E_\bullet, \varphi_\bullet$ is the Koszul complex of $M$ and
$F_\bullet, \psi_\bullet$ is the Koszul complex of a complete
intersection ideal $(z_1^{\beta_1},\ldots, z_n^{\beta_n})$ contained in $M$, it is not hard to
explicitly find mappings $a_k$ so that the diagram \eqref{lediagram}
commutes. 
Indeed, let $m_1,\ldots, m_r$ be a minimal set of generators of $M$,
ordered so that $m_j=z_j^{\alpha_j}$ for $j=1,\ldots, n$; note that
there are such generators since $M$ is Artinian. Identify the set of
generators with a section $\sum m_j e_j$ of a (trivial) rank $r$
bundle $\widetilde E$. Similarly identify $z_1^{\beta_1}, \ldots,
z_n^{\beta_n}$ with a section $\sum z_j^{\beta_j}\epsilon_j$ of a rank
$n$ bundle $\widetilde F$ and construct the Koszul complexes
$E_\bullet,\varphi_\bullet$ and $F_\bullet, \psi_\bullet$ as in
Example ~\ref{koszulex}. Now we can choose $a_{k-1}:\Lambda^k
\widetilde F^*\to \Lambda^k \widetilde E^*$ as the mapping 
$a_{k-1}:\epsilon_{i_1}^*\w\cdots\w \epsilon_{i_k}^*\mapsto
z_{i_1}^{\beta_{i_1}-\alpha_{i_1}}\cdots
z_{i_k}^{\beta_{i_k}-\alpha_{i_k}} e_{i_1}^*\w\cdots\w e_{i_k}^*$.  
Theorem~3.2 in \cite{L} then 
gives a formula relating the currents $R^E=R_{BM}^M$ and $R^F$, the latter given
by \eqref{svart}. However, when $M$ is not a complete intersection and
thus $E$ does not end at level $n-1$, the formula relating the currents is more involved than
\eqref{tricket}; there appears an extra term, which seems hard to
compute in general, see \cite[Equation~(3.2)]{L}.

\subsection{Non-Artinian monomial ideals}\label{higher} 
In \cite{W3} we also computed residue currents (up to nonvanishing factors) associated with cellular
resolutions of non-Artinian
monomial ideals. 

The method in this paper is not as well adapted to resolutions of
non-Artinian ideals. First, to be able to use the simple form \eqref{tricket}
of the comparison formula for residue currents it is important that $M$
is Cohen-Macaulay. 
Second, even if $M$ is Cohen-Macaulay, there is in general no such
natural (resolution of an) ideal to compare with as the monomial complete
intersection ideal $N=(z_1^{b_1},\ldots, z_n^{b_n})$ in the Artinian case.

\begin{ex}\label{cohen}
    Let $M$ be the ideal $M = (z_1z_2,z_1z_3,z_2z_3)$ in $S = \C[z_1,z_2,z_3]$.
Then  
\begin{equation}\label{nunu}
0 \longrightarrow S^{\oplus 2} 
\xrightarrow[]{ \left[ \begin{array}{cc} -z_3 & 0 \\ z_2 & -z_2 \\ 0 &
     z_1 \end{array} \right] }
S^{\oplus 3} 
\xrightarrow[]{ \left[ \begin{array}{ccc} z_1z_2 & z_1z_3 &
           z_2z_3 \end{array} \right]}
        S
\end{equation}
is a free resolution of $M$. 
Let $E_\bullet, \varphi_\bullet$ be the corresponding vector bundle
complex. 
Next, let $f$ be the regular sequence $f_1=z_1z_2,
f_2=(z_1+z_2)z_3$, and let $F_\bullet, \psi_\bullet$ be the Koszul
complex of $f$. Then it is not hard to explicitly find the morphisms
$a_1, a_0$, and $a_{-1}$. Since the ideals $M$ and $(f_1,f_2)$ are
Cohen-Macaulay we may apply the comparison formula
\eqref{tricket}. A computation gives 
\begin{multline*}
R^E= 
\frac{1}{z_1} \dbar \frac{1}{z_3} \w \dbar \frac{1}{z_2}
\left[\begin{array}{c} 1\\ 0 \end{array} \right]  +
\frac{1}{z_2} \dbar \frac{1}{z_3} \w \dbar \frac{1}{z_1}
\left[\begin{array}{c} 1 \\ 1 \end{array} \right] +  
 \frac{1}{z_3} \dbar\frac{1}{z_2} \w \dbar \frac{1}{z_1}
 \left[\begin{array}{c} 0 \\ 1\end{array} \right].
    \end{multline*}
Observe that $R$ is not symmetric in $z_1$ and $z_2$, although the ideal
$M$ is. This is, however, not too surprising, since the resolution
\eqref{nunu} is not symmetric in $z_1$ and $z_2$. 
\end{ex}

A general strategy for computing the residue current associated with the
resolution $E_\bullet,\varphi_\bullet$ of a
(monomial) Cohen-Macaulay ideal $M$ of codimension $p$ is to look for
a regular sequence $f_1,\ldots, f_p$ contained in $M$ and then apply the comparison
formula \eqref{tricket} to $E_\bullet,\varphi_\bullet$ and the Koszul complex $F_\bullet,
\psi_\bullet$ of $f$. One way of finding such a regular sequence is to consider $p$ sufficiently generic linear
combinations $f_1,\ldots, f_p$ of the 
generators of $M$, as was done in Example ~\ref{cohen}. 
However, when
the $f_j$ are not monomials the computation of the
 current $R^F=R_{CH}^f$ can become much more involved. Also,
although the complex $F_\bullet,\psi_\bullet$ is simple, it
may be hard to find the morphism $a_k$ in general.

If $E_\bullet,\varphi$ is a resolution of a non-Cohen-Macaulay
ideal, the comparison formula in \cite{L} is more involved than
\eqref{tricket}. For computations of residue currents in this case,
see \cite[Section~5]{L}.

\section{Relations to fundamental cycles}\label{fundamental}

Our original motivation for computing the coefficients $c_\sigma$ of
the entries \eqref{gamla} of 
$R^X$ was that we wanted to understand the current
\begin{equation}\label{dfr}
    D\varphi \circ R:= D\varphi_0\circ\cdots\circ D\varphi_{p-1}\circ R, 
\end{equation}
when $R=R^E$ is the residue current associated with a resolution
\eqref{karv} of an
ideal sheaf $\a$ of codimension $p$ and $D$ is the connection on $\End
E$ induced by connections on $E=\bigoplus E_k$.

Let $\a$ be a complete intersection ideal, defined by a regular
sequence $f_1,\ldots, f_p$ and let \eqref{bunt} be the Koszul complex of
$f_j$, see Example ~\ref{koszulex}, equipped with the trivial metrics 
so that $D$ is the trivial connection $d$.  
Then \eqref{dfr} equals $p!$ times the current 
\begin{equation}\label{dfch}
R^f_{CH}\w df_1\w\cdots \w df_p= (2\pi i)^p[\a],
\end{equation}
where $[\a]$ is the current of integration along the \emph{fundamental
  cycle} of $\a$. The equality \eqref{dfch} was proved in \cite{CH}. 
Recall that for an Artinian ideal $\a \subseteq \Ok_0^n$, the fundamental cycle
of $\a$ is $[\a] = m[0]$, where $m=\dim_\C\Ok^n_0/\a$ is the \emph{geometric multiplicity}
of $\a$. For an arbitrary ideal $\a$, with irreducible components $Z_i$ (i.e., irreducible
components of the radical ideal of $\a$), the fundamental cycle of $\a$ is $[\a] = \sum m_i [Z_i]$
where $m_i$ are the geometric multiplicities of $\a$ along $Z_i$.
The geometric multiplicity $m_i$ of $\a$ along $Z_i$ can be defined as the geometric
multiplicity of the Artinian ideal $\a + \b$, where $\b$ is the ideal of a generic
smooth variety transversal to $Z_i$. For more details regarding fundamental cycles,
see \cite[Section 1.5]{F}.

Using the comparison formula for residue currents from \cite{L}, 
we recently managed to prove that 
\begin{equation}\label{monique}
D\varphi\circ R=p! (2\pi i)^p [\a]
\end{equation}
for any resolution \eqref{karv} of any equidimensional ideal (i.e., all
minimal primes are of the same dimension) $\a\subset \O_0^n$, 
thus generalizing \eqref{dfch}. 
This factorization of the fundamental cycle is closely related to a
result by Lejeune-Jalabert, \cite{L-J}, who proved a cohomological
version of \eqref{monique} for Cohen-Macaulay ideals, and it will be
the subject of the forthcoming paper \cite{LW}.

For the residue current associated with the hull resolution of a
generic Artinian monomial ideal we can give an alternative proof of
\eqref{monique} (with the trivial connection $d$) using Theorem~\ref{main}. In fact, we get a 
refinement of \eqref{monique}:  For each permutation $s_1,\ldots, s_n$
of $1,\ldots, n$, 
\begin{equation}\label{forfinad}
\frac{\partial f_1}{\partial z_{s_1}} ~dz_{s_1}\circ \cdots \circ \frac{\partial
  f_n}{\partial z_{s_n}} ~dz_{s_n} \circ R = c_n (2\pi i)^n[\a],  
\end{equation} 
where $c_n=(-1)^{n^2}\cdot (-1)^{\frac{n(n-1)}{2}}$. For an
explanation of why the constant $c_n$ appears in the right hand side of
\eqref{forfinad}, but not in \eqref{monique}, see \cite{LW}. 
We will show how this works when $n=2$. For $n\geq 3$, the computation
of $d\varphi$ gets more involved; the general case will therefore be
treated in the separate paper \cite{W5}.

First, let us describe the geometric multiplicity $\dim_\C\Ok_0^n/M$ of a monomial ideal $M\subset\Ok_0^n$. 
Let $\R_+$ denote the nonnegative real numbers and let 
$T_M$ be the \emph{staircase} 
$\R^n_+\setminus \bigcup_{z^\alpha\in M}\{\alpha+\R^n_+\}$ of $M$. If
$M$ is Artinian, then $T_M$ is a bounded set in $\R^n_+$. The name
staircase is motivated by the shape of $T_M$. 
If $n=2$ each Artinian monomial ideal $M$
is of the form $M=(z^{a_1}w^{b_1}, \ldots, z^{a_r}w^{b_r})$ for some integers 
$a_1>\ldots >a_r=0$ and $0=b_1<\ldots <b_r$. Then $T_M$ looks like a
staircase 
with inner corners $(a_j,b_j)$ and outer corners $(a_j, b_{j+1})$,
see Figure ~\ref{trappor}. 
\begin{figure}
\begin{center}
\includegraphics{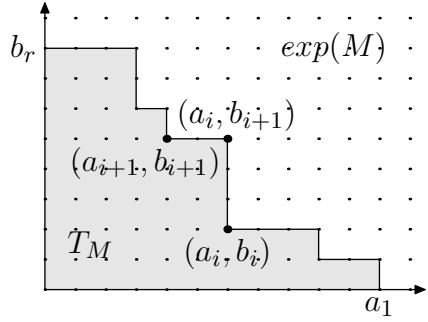}
\caption{The staircase $T_M$ of an Artinian monomial ideal in $\C^2$. The lattice points above $T_M$ are the exponents $\text{exp}(M)$ of monomials in $M$.}
\label{trappor}
\end{center}
\end{figure}
In general there is an ``inner corner'' $\alpha$  
for each minimal generator $z^\alpha$ of $M$ and one ``outer
corner'' $\alpha$ for each intersectand $\m^\alpha$ in the
irredundant irreducible decomposition. 
 If $M$ is generic, there is a one-to-one correspondence between faces $\sigma\in \hull(M)_{n-1}$, with labels $m_{\sigma}=z^{\alpha_\sigma}$, and outer corners $\alpha$ in $T_M$. 
The points in $\Z^n\cap T_M$ are precisely the exponents of monomials that are not in $M$. In other words, 
$\Ok_0^n/M =\spann_\C\{z^\alpha\mid \alpha\notin T_M\}$. It follows
that $\dim_\C\Ok_0^n/M=\Vol(T_M)$, where $\Vol$ is the usual Euclidean
volume in $\R^n$.

Now assume that $n=2$, and that $M$ is an Artinian ideal, minimally
generated by $z^{a_i}w^{b_i}$, $a_1>\ldots >a_r=0$ and $0=b_1 <\ldots
<b_r$. Then $\hull (M)$ is one-dimensional, with one vertex $v_i$ for
each generator $z^{a_i}w^{b_i}$ and one edge $\sigma_i$, with label
$z^{a_i}w^{b_{i+1}}$, for each outer corner $(a_i,b_{i+1})$ in $T_M$. 
The mappings in $\F_{\hull(M)}$ are given by 
$\varphi_0: e_{v_i}\mapsto z^{a_i}w^{b_i}e_\emptyset$ and 
$\varphi_1: e_{\sigma_i}\mapsto z^{a_i-a_{i+1}} e_{v_{i+1}}-
w^{b_{i+1}-b_i} e_{v_i}$ and 
by Theorem~\ref{main}, 
\[
R=R^{\hull(M)}=\sum_{i=1}^{r-1}\dbar\left
  [\frac{1}{w^{b_{i+1}}}\right ]\w \dbar\left [\frac{1}{z^{a_i}}\right ] e_{\sigma_i}\otimes e^*_\emptyset.
\]

Let us compute $\frac{\partial\varphi_0}{\partial z}dz\circ
\frac{\partial\varphi_1}{\partial w}dw\circ R$. 
Note that 
\[
\frac{\partial\varphi_0}{\partial z} dz = 
\sum_{i=1}^r a_i z^{a_i}w^{b_i} \frac{dz}{z} e_{v_i}^*\otimes e_\emptyset
\]
and
\[
\frac{\partial\varphi_1}{\partial w} dw = 
- \sum_{i=1}^{r-1} (b_{i+1}-b_i) w^{b_{i+1}-b_i} \frac{dw}{w}
 e_{\sigma_i}^*\otimes e_{v_i}, 
\]
so that 
\begin{equation*}
-\frac{\partial\varphi_0}{\partial z}dz\circ \frac{\partial\varphi_1}{\partial w}dw=
\sum_{i=1}^{r-1} a_i(b_{i+1}-b_i)z^{a_i} w^{b_{i+1}}\frac{dz}{z}\w\frac{dw}{w} e_{\sigma_i}^*\otimes e_\emptyset. 
\end{equation*}
Let $P_i=\{x\in T_M \mid 0 \leq x_1 < a_i, b_i \leq x_2 < b_{i+1} \}$
for $i = 1,\dots,r-1$. 
Then the $P_i$ form a partition of $T_M$, cf. Figure
\ref{partitioner} and, in particular, $\Vol(T_M)=\sum\Vol(P_i)$. 
Note that $\Vol(P_i)=a_i(b_{i+1}-b_i)$. Hence (identifying $e^*_\emptyset \otimes e_\emptyset$ with
$1$)
\begin{multline*}
-\frac{\partial\varphi_0}{\partial z}dz\circ \frac{\partial\varphi_1}{\partial w}dw\circ R=
\sum_{i=1}^{r-1} \Vol(P_i) z^{a_i} w^{b_{i+1}}\frac{dz}{z} \w
\frac{dw}{w}\w 
\dbar\left
  [\frac{1}{w^{b_{i+1}}}\right ] \w \dbar\left [\frac{1}{z^{a_i}}\right ]= 
\\
\sum _{i=1}^{r-1} \Vol(P_i) ~\dbar\left [\frac{1}{w}\right
]\w\dbar\left [\frac{1}{z}\right ] \w dz\w dw =
(2\pi i)^2 \Vol(T_M) [0],  
\end{multline*}
so we have proved \eqref{forfinad} (for $z_{s_1}=z$ and $z_{s_2}=w$).

\begin{figure}
\begin{center}
\includegraphics{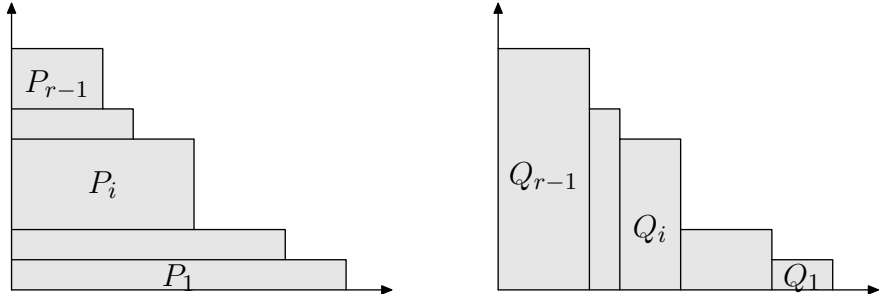}
\caption{Partitions of $T_M$ as rectangles $P_i$ and rectangles
  $Q_i$.}
\label{partitioner}
\end{center}
\end{figure}

By similar arguments we get that 
$-\frac{\partial\varphi_0}{\partial w}dw\circ
\frac{\partial\varphi_1}{\partial z}dz\circ R= \sum _{i=1}^{r-1}
\Vol(Q_i) (2\pi i)^2[0]$, where $Q_i= \{x\in
T_M \mid a_{i+1} \leq x_1 < a_i, 0 \leq x_2 < b_{i+1} \}$ for
$i=1,\ldots, r-1$, see Figure ~\ref{partitioner}. 
Again, the rectangles $Q_i$ form a partition of $T_M$ and thus
\eqref{forfinad} holds also for this permutation ($z_{s_1}=w$ and
$z_{s_2}=z$) of the variables. To conclude, we have proved
\eqref{monique} for hull resolutions of monomial ideals in dimension
$2$ with $D=d$.

For a generic Artinian monomial ideal $M\subset \Ok_0^n, n\geq 3$
one can analogously define cuboids $P_{\alpha, s}$, where $\alpha$ is
an outer corner of $T_M$ and $s$ is a permutation $s_1, \ldots, s_n$ of $1,\ldots, n$, such that for a fixed $s$, $\{P_{\alpha,s}\}_\alpha$ defines a partition of $T_M$ and moreover
\begin{equation*}
\frac{\partial\varphi_0}{\partial z_{s_1}}dz_{s_1}\circ \cdots\circ 
\frac{\partial\varphi_{n-1}}{\partial z_{s_n}}dz_{s_n} =
\sum_{\sigma\in\hull(M)_{n-1}} \Vol(P_{\alpha_\sigma, s}) z^{\alpha_\sigma} \frac{dz_1}{z_1}\w\cdots\w\frac{dz_n}{z_n} e_{\sigma}^*\otimes e_\emptyset.
\end{equation*}
Together with Theorem~\ref{main} this proves \eqref{forfinad} and thus
\eqref{monique} in this case. 
However, for $n\geq 3$, the construction of the $P_{\alpha, s}$ is
more delicate than that of $P_i$ and $Q_i$, see \cite{W5}.

\def\listing#1#2#3{{\sc #1}:\ {\it #2},\ #3.}


\begin{thebibliography}{9999}

\bibitem{A}\listing{M.\ Andersson}
{Residue currents and ideals of holomorphic functions}
{Bull.\ Sci.\ Math. {\bf 128} (2004) no. 6, 481--512}

\bibitem{AS}\listing{M.\ Andersson \& H.\ Samuelsson}
{A Dolbeault-Grothendieck lemma on complex spaces via Koppelman
  formulas} 
{Invent. Math. {\bf 190} (2012), no. 2, 261--297}


\bibitem{AS2}\listing{M.\ Andersson \& H.\ Samuelsson}
{Weighted Koppelman formulas and the $\overline\partial$-equation on
  an analytic space}
{J.\ Funct.\ Anal.  {\bf 261}  (2011), 777--802} 


\bibitem{ASS}\listing{M.\ Andersson \& H.\  Samuelsson \& J.\ Sznajdman}
{On the Brian\c con-Skoda theorem on a singular variety}
{Ann.\ Inst.\ Fourier {\bf 60} (2010), 417--432}

\bibitem{AW}\listing{M.\ Andersson \& E.\ Wulcan}
{Residue currents with prescribed annihilator ideals}
{Ann.\ Sci.\ École Norm.\ Sup. {\bf 40} (2007) no. 6, 985--1007}

\bibitem{AW2}\listing{M.\ Andersson \& E.\ Wulcan}
{Decomposition of residue currents}
{J.\ Reine Angew.\ Math. {\bf 638} (2010), 103--118}

\bibitem{AW4}\listing{M.\ Andersson \& E.\ Wulcan}
{On the effective membership problem on singular varieties}
{Preprint, arXiv:1107.0388}


\bibitem{BaS}\listing{D.\ Bayer \& B.\ Sturmfels}
{Cellular resolutions of monomial modules}
{J.\ Reine Angew.\ Math. {\bf 502} (1998) 123--140}


\bibitem{BPS}\listing{D.\ Bayer \& I.\ Peeva \& B.\ Sturmfels}
{Monomial resolutions}
{Math.\ Res.\ Lett.\  {\bf 5}  (1998),  no. 1-2, 31--46}

\bibitem{CH}\listing{N.r.\ Coleff \& M.e.\ Herrera}
{Les courants r\'esiduels associ\'es \`a une forme m\'eromorphe}
{Lect. Notes in Math. {\bf 633},  Berlin-Heidelberg-New York (1978)}


\bibitem{DS}\listing{A.\ Dickenstein  \& C.\ Sessa}
{Canonical representatives in moderate cohomology}
{Invent. Math. {\bf 80} (1985),  417--434}


\bibitem{E}\listing{D.\ Eisenbud}
{Commutative algebra. With a view toward algebraic geometry}
{Graduate Texts in Mathematics, 160. Springer-Verlag, New York, 1995}{}{}{}


\bibitem{E2}\listing{D.\ Eisenbud}
{The geometry of syzygies. A second course in commutative algebra and algebraic geometry}
{Graduate Texts in Mathematics, 229. Springer-Verlag, New York, 2005}

\bibitem{F}\listing{W.\ Fulton}
{Intersection theory. Ergebnisse der Mathematik und ihrer Grenzgebiete}
{Springer-Verlag, Berlin, 1984}


\bibitem{L}\listing{R.\ L\"ark\"ang} 
{A comparison formula for residue currents}
{Preprint, arXiv:1207.1279}

\bibitem{LW}\listing{R.\ L\"ark\"ang \& E.\ Wulcan} 
{Residue currents and fundamental cycles}
{In preparation}


\bibitem{L-J}\listing{M.\ Lejeune-Jalabert} 
{Remarque sur la classe fondamentale d'un cycle}
{C.\ R.\ Acad.\ Sci.\ Paris Sér. I Math.  {\bf 292}  (1981), no. 17, 801--804}


%\bibitem{Matsu}\listing{H.\ Matsumura}
%{Commutative ring theory}  
%{Cambridge Studies in Advanced Mathematics, {\bf 8}}

\bibitem{MS}\listing{E.\ Miller \& B.\ Sturmfels}
{Combinatorial commutative algebra} 
{Graduate Texts in Mathematics {\bf 227} Springer-Verlag, New York, 2005}

\bibitem{MSY}\listing{E.\ Miller \& B.\ Sturmfels \& K.\ Yanagawa}
{Generic and cogeneric monomial ideals, Symbolic computation in algebra, analysis, and geometry (Berkeley, CA, 1998)}
{J.\ Symbolic Comput. {\bf 29} (2000) no 4--5, 691--708}


\bibitem{P}\listing{M.\ Passare}
{Residues, currents, and their relation to ideals of holomorphic functions} 
{Math.\ Scand.\ {\bf 62} (1988), no. 1, 75--152}

\bibitem{P2}\listing{M.\ Passare}
{A calculus for meromorphic currents}
{J.\ Reine Angew.\ Math. {\bf 392} (1988), 37--56}


\bibitem{PTY}\listing{M.\ Passare \& A.\ Tsikh \&  A.\ Yger}
{Residue currents of the Bochner-Martinelli type}
{Publ.\ Mat.  {\bf 44} (2000), 85--117}

\bibitem{Sp}\listing{E.\ Spanier}
{Algebraic topology}
{McGraw-Hill Book Co., New York-Toronto, Ont.-London 1966} 


\bibitem{Sz}\listing{J.\ Sznajdman}
{A residue calculus approach to the uniform Artin-Rees lemma}
{Israel J. Math., to appear} 

\bibitem{Ta}\listing{J.\ L.\ Taylor}
{Several complex variables with connections to algebraic geometry and Lie groups}
{Graduate Studies in Mathematics, {\bf 46} American Mathematical Society, Providence, RI, 2002}

\bibitem{W}\listing{E.\ Wulcan}
{Residue currents of monomial ideals}
{Indiana Univ.\ Math.\ J. {\bf 56} (2007), no.\ 1, 365--388}

\bibitem{W3}\listing{E.\ Wulcan}
{Residue currents constructed from resolutions of monomial ideals}
{Math.\ Z. {\bf 262} (2009), 235--253} 

\bibitem{W4}\listing{E.\ Wulcan}
{On weighted Bochner-Martinelli residue currents} 
{Math.\ Scand., {\bf 110} (2012), 18--34}

\bibitem{W5}{{\sc E. Wulcan}:\ In preparation.}

\bibitem{Z}\listing{G.\ Ziegler}
{Lectures on polytopes}
{Graduate Texts in Mathematics {\bf 152} Springer Verlag, New York, 1995}

\end{thebibliography}
\end{document}